\documentclass{amsart}

\usepackage{amssymb}
\usepackage{amsmath}
\newtheorem{theorem}{Theorem}
\newtheorem{lemma}[theorem]{Lemma}
\newtheorem{corollary}[theorem]{Corollary}

\theoremstyle{definition}
\newtheorem{remark}[theorem]{Remark}

\newcommand{\N}{\mathbb{Z}_{>0}}
\newcommand{\Z}{\mathbb{Z}}
\newcommand{\Q}{\mathbb{Q}}

\newcommand{\F}{\mathbb{F}}

\newcommand{\set}[2]{\{#1 \: | \; #2 \}}
\newcommand{\rad}{\mathrm{rad}}
\newcommand{\Norm}{\mathrm{Norm}}
\newcommand{\Ideal}[1]{\mathfrak{#1}}
\newcommand{\PP}{\mathbb{P}}
\newcommand{\E}{\mathfrak{E}}
\newcommand{\ord}{\mathrm{ord}}

\title[A refined modular approach to $x^2+y^{2n}=z^3$]%
{A refined modular approach to the Diophantine equation
$x^2+y^{2n}=z^3$}

\author{Sander R. Dahmen}

\email{dahmen@math.ubc.ca}

\address{Department of Mathematics\\
The University of British Columbia\\ Room 121, 1984 Mathematics
Road\\ Vancouver, B.C.\\ Canada, V6T 1Z2}



\date{\today}

\begin{document}

\maketitle

\begin{abstract}
Let $n$ be a positive integer and consider the Diophantine
equation of generalized Fermat type $x^2+y^{2n}=z^3$ in nonzero
coprime integer unknowns $x,y,z$. Using methods of modular forms
and Galois representations for approaching Diophantine equations,
we show that for $n \in \{5, 31\}$ there are no solutions to this
equation. Combining this with previously known results, this
allows a complete description of all solutions to the Diophantine
equation above for $n \leq 10^7$. Finally, we show that there are
also no solutions for $n\equiv -1 \pmod{6}$.
\end{abstract}

\section{Introduction}\label{sec Introduction}

Since the proof of Fermat's Last Theorem \cite{Wiles 95},
\cite{Taylor Wiles 95} and the establishment of the full
Shimura-Taniyama-Weil conjecture \cite{BCDT 01}, many Diophantine
equations were solved using a strategy similar to that of the
proof of FLT. Amongst them are various so-called generalized
Fermat equations. These are Diophantine equations of the form
\begin{equation}\label{eqn generalized Fermat}
ax^p+by^q=cz^r, \quad x,y,z \in \Z, \quad xyz\not=0, \quad
\gcd(x,y,z)=1
\end{equation}
where $a,b,c$ are nonzero integers and $p,q,r$ are integers $\geq
2$. The nature of the solutions depends very much on the quantity
\[\chi(p,q,r):=\frac{1}{p}+\frac{1}{q}+\frac{1}{r}-1.\]
If $\chi(p,q,r)>0$, then there are either no solutions to
(\ref{eqn generalized Fermat}) or infinitely many. In the latter
case, there exist finitely many triples $(X,Y,Z)$ of binary forms
$X,Y,Z \in \Q[u,v]$ satisfying $aX^p+bY^q=cZ^r$ such that every
solution $(x,y,z)$ to (\ref{eqn generalized Fermat}) can be
obtained by specializing the variables $u,v$ to integers for one
of these triples; see \cite{Beukers 98}. If $\chi(p,q,r)=0$, then
the determination of the solutions to (\ref{eqn generalized
Fermat}) basically boils down to finding rational points on curves
of genus one. If $\chi(p,q,r)<0$, then there exists a curve $C$ of
genus $\geq 2$ and a covering $\phi: C \to \PP^1$, both defined
over a number field $K$, such that for every solution $(x,y,z)$ to
(\ref{eqn generalized Fermat}) we have $[ax^p:cz^r] \in
\phi(C(K))$. Since by Faltings' theorem $C(K)$ is finite, there
are only finitely many solutions to (\ref{eqn generalized Fermat})
in this case; see \cite{Darmon Granville 95}.

Of special interest is the generalized Fermat equation with
$a=b=c=1$. In the case $\chi(p,q,r)>0$, all the (parameterized)
solutions are known; see \cite{Beukers 98} and \cite{Edwards 04}.
If $\chi(p,q,r)=0$, the only solution, up to sign and permutation,
is given by the Catalan solution $2^3+1^6=3^2$. If
$\chi(p,q,r)<0$, the only solutions known, up to sign and
permutation, are given by $2^3+1^q=3^2 \ (q \geq 7)$ and
\begin{multline*}
17^3+2^7=71^2, \quad 2213459^2+1414^3=65^7,\\
15312283^2+9262^3=113^7, \quad 76271^3+17^7=21063928^2;
\end{multline*}
\[1549034^2+33^8=15613^3,\quad 96222^3+43^8=30042907^2;\]
\[13^2+7^3=2^9;\]
\[7^2+2^5=3^4,\quad 11^4+3^5=122^2.\]
For a list containing many (families of) triples $(p,q,r)$ for
which (\ref{eqn generalized Fermat}) has been solved in this case,
we refer to \cite[Table 1]{PSS 07}. In this paper we will focus on
the special case $(p,q,r)=(2,2n,3)$ for $n \in \Z_{>0}$, i.e. we
are concerned with the Diophantine equation
\begin{equation}\label{eqn 2_2l_3}
x^2+y^{2n}=z^3, \quad x,y,z \in \Z, \quad xyz \not=0, \quad
\gcd(x,y,z)=1.
\end{equation}

\subsubsection*{Previously known results.}

For $n=1,2$ we have $\chi(2,2n,3)>0$, and in both cases there are
infinitely many solutions to (\ref{eqn 2_2l_3}). By factoring
$x^2+y^2=z^3$ over the Gaussian integers as $(x+iy)(x-iy)=z^3$ one
readily gets, that a solution to (\ref{eqn 2_2l_3}) for $n=1$
satisfies
\[(x,y,z)=(u(u^2-3 v^2), v (3 u^2-v^2),u^2+v^2)\]
for some $u,v \in \Z$ with $\gcd(u,v)=1$ and $uv\not=0$. By
demanding that $v (3u^2-v^2)$ is a square, one can obtain
parameterized solutions to (\ref{eqn 2_2l_3}) for $n=2$. This was
carried out by Zagier and reported in \cite{Beukers 98}; up to
sign there are $4$ parameterizations, $3$ of which have
coefficients in $\Z$. In \cite{Edwards 04} parameterized solutions
to (\ref{eqn 2_2l_3}) for $n=2$ were obtained by a different
method, the same parameterizations as in \cite{Beukers 98} were
found, except that the one with non-integer coefficients was
replaced by one with coefficients in $\Z$.

For $n=3$ there are no solutions to (\ref{eqn 2_2l_3}). This
follows readily from the well-known fact that the only rational
points on the elliptic curve given by $Y^2=X^3-1$ are
$(X,Y)=(1,0)$ and the point at infinity.

By demanding that $y$ is a square for the parameterized solutions
to (\ref{eqn 2_2l_3}) for $n=2$, one obtains genus $2$ curves such
that every solution to (\ref{eqn 2_2l_3}) for $n=4$ corresponds to
a rational point on one of these curves. This, together with the
determination of all rational points on these curves using
effective Chabauty methods, was carried out in \cite{Bruin 99}.
The result is that, up to sign, the only solution to (\ref{eqn
2_2l_3}) for $n=4$ is given by $1549034^2+33^8=15613^3$.

We see that it suffices to deal with $n$ a prime $> 3$. In
\cite{Chen 08}, the equation is studied using methods of modular
forms and Galois representations for approaching Diophantine
equations. In particular, an explicit criterion is given for
showing that (\ref{eqn 2_2l_3}) has no solution for a given prime
$n >7$. This criterion is verified for all primes $7 < n < 10^7$
except $n=31$, thereby showing that (\ref{eqn 2_2l_3}) has no
solutions for these values of $n$. The nonexistence of solutions
for $n=7$ follows from the work of \cite{PSS 07}, so the only
small values of $n$ left to deal with are $n=5,31$.

At this point we should mention that, for $m\in \Z_{>1}$, the
Diophantine equation
\[x^2+y^3=z^m, \quad x,y,z \in \Z, \quad xyz \not=0, \quad
\gcd(x,y,z)=1\] is much harder to deal with than (\ref{eqn
2_2l_3}) is (apart from some small values of $m$), even if we just
consider even $m$. This is because of the Catalan solution
$3^2+(-2)^3=1^m$. For $m \leq 5$ there are infinitely many
solutions. The parameterization for $m=2$ is again very easy to
obtain, for $m=3,4,5$ we refer to \cite{Edwards 04} ($m=3$ was
earlier done by Mordell and $m=4$ by Zagier). The case $m=6$ is
classical, it amounts to determining the rational points on the
elliptic curve given by $Y^2=X^3+1$. The cases $m=7,8,9,10$ are
solved in \cite{PSS 07}, \cite{Bruin 03}, \cite{Bruin 05},
\cite{Siksek} respectively.

\subsubsection*{New results.}

In this paper, we will extend the criterion mentioned above to all
primes $n>3$ and use it to show that there are no solutions for
$n=5$. This is basically done by showing that a Frey curve
associated to a (hypothetical) solution has irreducible mod-$n$
representation for $n=5,7$. By using extra local information
obtained from classical algebraic number theory, we obtain a
refined criterion, which is used to show that (\ref{eqn 2_2l_3})
has no solutions for $n=31$. So from a Diophantine point of view
our main result is the following.

\begin{theorem}\label{thm main}
If $n \in \{5,31\}$, then (\ref{eqn 2_2l_3}) has no solutions.
\end{theorem}

Although we focus on very specific equations, we feel that the
methods we use to overcome the earlier difficulties, could
definitely be used in other cases as well. Combining our main
result with the previously known results mentioned above, we
arrive at a description of all solutions for $n \leq 10^7$.

\begin{corollary}\label{cor main}
Let $n \in \N$ with $n \leq 10^7$. If $n \in \{1,2,4\}$, then
(\ref{eqn 2_2l_3}) has solutions, all of which are described
above. If $n \not\in \{1,2,4\}$, then (\ref{eqn 2_2l_3}) has no
solutions.
\end{corollary}

Finally, by applying a key idea from \cite{Chen Siksek 09}, we
solve (\ref{eqn 2_2l_3}) for infinitely many (prime) values of
$n$.

\begin{theorem}\label{thm n congruence}
If $n \in \N$ with $n \equiv -1 \pmod{6}$, then (\ref{eqn 2_2l_3})
has no solutions.
\end{theorem}

\section{A modular approach to $x^2+y^{2l}=z^3$}\label{sec A modular
approach}

Throughout this section, let $l$ denote a prime $>3$. We shall
explain how modular forms and Galois representations can be used
to study (\ref{eqn 2_2l_3}) for $n=l$. We want to stress that the
methods and results described in this section are not new and can
be found essentially in \cite{Chen 08}. For a general introduction
to using methods of modular forms and Galois representations for
approaching Diophantine equations, one could consult e.g. Chapter
2 of the author's Ph.D. thesis \cite{Dahmen 08}.

\begin{lemma}\label{lem descent}
Suppose that $(x,y,z)$ is a solution to (\ref{eqn 2_2l_3}) for
$n=l$. Then
\begin{equation}\label{eqn descent 1}
(x,y^l,z)=(u(u^2-3 v^2), v (3 u^2-v^2),u^2+v^2)
\end{equation}
for some $u,v \in \Z$ with $\gcd(u,v)=1,\ uv\not=0,\ 3|v$ and
$2|uv$. Furthermore,
\begin{equation}\label{eqn descent 2}
3v=r^l \quad \mathrm{and} \quad 3u^2-v^2=3s^l
\end{equation}
for some $r,s \in \Z$ with $\gcd(r,s)=1$ and $rs \not= 0$.
\end{lemma}

\begin{proof}
That (\ref{eqn descent 1}) holds for some nonzero coprime integers
$u,v$ follows immediately from the parameterization of solutions
to (\ref{eqn 2_2l_3}) for $n=1$, given in Section \ref{sec
Introduction}.
From $v(3 u^2-v^2)=y^l$ we see that up to primes dividing
$\gcd(v,3u^2-v^2)$, both $v$ and $3u^2-v^2$ must be $l$-th powers.
One easily checks that $\gcd(v,3u^2-v^2)$ equals either 1 or 3.

First suppose that $\gcd(v,3u^2-v^2)=1$. Then $v=r^l$ and
$3u^2-v^2=s^l$ for some nonzero $r,s \in \Z$. Furthermore, $r,s$
and $u$ are pairwise coprime and $(r^2)^l+s^l=3u^2$. However,
\cite[Theorem 1.1]{Bennett Skinner 04} tells us that for an
integer $n>3$ there are no nonzero (pairwise) coprime integers
$a,b,c$ satisfying $a^n+b^n=3c^2$, which contradicts
$\gcd(v,3u^2-v^2)=1$.
So suppose now that $\gcd(v,3u^2-v^2)=3$. Then $3 |v$ and of
course $3 \nmid u$, so $3||3u^2-v^2$. We get that $3v=r^l$ and
$3u^2-v^2=3s^l$ for some nonzero $r,s \in \Z$.
Finally, if both $u,v$ are odd, then $v (3u^2-v^2) \equiv 2
\pmod{4}$, but this contradicts that $v (3u^2-v^2)$ is an $l$-th
power.
\end{proof}

We have reduced (\ref{eqn 2_2l_3}) for $n=l$ to the equation $v (3
u^2-v^2)=y^l$ with $u,v,y$ as in the lemma above. This equation
can be approached using methods of modular forms and Galois
representations. In fact, one can use almost exactly the same
methods as used in \cite{Kraus 98} for studying the equation
$a^3+b^3=c^l$ in nonzero coprime integers $a,b,c$.

Suppose $(x,y,z)$ is a solution to (\ref{eqn 2_2l_3}) for $n=l$,
let $u,v,r,s$ be as in Lemma \ref{lem descent} and consider the
following Frey curve associated to this solution
\begin{equation}\label{eqn Frey Curve}
E: Y^2=\left\{
\begin{array}{ll}
X^3+2uX^2+\frac{v^2}{3}X & \mbox{if $u$ is even}; \\
X^3 \pm uX^2+\frac{v^2}{12}X,\ \pm u \equiv 1 \pmod{4} & \mbox{if
$u$ is odd}.
\end{array}
\right.
\end{equation}
Because $3|v$ and $2|v$ if $u$ is odd, we see that the model for
$E$ above actually has coefficients in $\Z$. Using Tate's
algorithm, or the in practice very handy to use \cite{Papadopoulos
93}, one finds that the model for $E$ is minimal at every prime
$p$, except at $p=2$ when $u$ is odd. Furthermore, the conductor
$N$ and the minimal discriminant $\Delta$ of $E$ satisfy (see also
\cite[Propositio 8]{Chen 08} and \cite[Lemma 2.1]{Bennett Skinner
04})
\begin{align*}
\Delta = & 2^{\alpha} \cdot 3^{-3} v^4(3u^2-v^2)=2^{\alpha} \cdot
3^{-6}
\left(r^4s\right)^l, \quad \alpha:=6 \mathrm{\ if\ } 2|u, \quad \alpha:=-12 \mathrm{\ if\ } 2\nmid u\\
N = & 2^{\beta} \cdot 3\, \rad_{\{2,3\}}(\Delta)=2^{\beta} \cdot
3\, \rad_{\{2,3\}}(rs), \quad \beta:=5 \mathrm{\ if\ } 2|u, \quad
\beta:=1 \mathrm{\ if\ } 2\nmid u,
\end{align*}
where for a finite set of primes $S$ and a nonzero $n \in \Z$,
$\rad_S(n)$ denotes the product of prime $p$ with $p|n$ and $p
\not\in S$.

\begin{remark}
From $x^2-z^3=-y^{2l}$, an obvious choice, up to quadratic twist,
for a Frey curve associated to the solution $(x,y,z)$ would be
\[E': Y^2=X^3-3zX+2x.\]
This model has (not necessarily minimal) discriminant $-2^6 \cdot
3^3(x^2-z^3)=2^6 \cdot 3^3 y^{2l}$, which has no primes $>3$ in
common with $c_4=2^4 \cdot 3^2 z$. Using (\ref{eqn descent 1}), we
can write $E'$ as
\[E': Y^2=X^3-3(u^2+v^2)X+2u(u^2-3 v^2),\]
from which we obtain that $E'$ has a rational $2$-torsion point
because the right-hand side factors as $(X+2 u) (X^2 - 2 u X + u^2
- 3 v^2 )$. The $2$-isogenous elliptic curve associates to the
$2$-torsion point $(X,Y)=(-2u,0)$ is simply, up to quadratic
twist, the Frey curve $E$ above. For both theory and practical
computation, it makes no difference whether $E$ or $E'$ is used,
except for establishing irreducibility of the mod-$5$
representation, where the computations actually become cleaner if
one uses $E'$ instead of $E$. The only reason for introducing $E$
above, is to follow \cite{Chen 08} more closely.
\end{remark}

For an elliptic curve $\E$ over $\Q$, let $\rho_l^{\E}:
\mathrm{Gal}(\overline{\Q}/\Q) \to \mathrm{Gl}_2(\F_l)$ denote the
standard $2$-dimensional mod-$l$ Galois representation associated
to $\E$ (induced by the natural action of
$\mathrm{Gal}(\overline{\Q}/\Q)$ on the $l$-torsion points $E[l]$
of $\E$). Suppose that $\rho_l^E$ is irreducible. Then by
modularity \cite{BCDT 01} and level lowering \cite{Ribet 90} (note
that $\rho_l^E$ is finite at $l$), we obtain that $\rho_l^E$ is
modular of level $N_0:=2^{\beta} \cdot 3$, weight $2$ and trivial
character. This means that $\rho_l^E \simeq \rho_l^f$ for some
(normalized) newform $f \in S^2(\Gamma_0(N_0))$, where of course
$\rho_l^f$ denotes the standard $2$-dimensional mod-$l$ Galois
representation associated to $f$. If $u$ is odd, then $N_0=6$ and
since there are no newforms in $S^2(\Gamma_0(6))$ we have reached
a contradiction in this case. If $u$ is even, then $N_0=96$ and
$S^2(\Gamma_0(96))$ contains $2$ newforms, both rational and
quadratic twists over $\Q(i)$ of each other. So $\rho_l^E \simeq
\rho_l^f \simeq \rho_l^{E_0}$ for some elliptic curve $E_0$ over
$\Q$ with conductor $N_0=96$. As is well-known, we get from
$\rho_l^E \simeq \rho_l^{E_0}$ by comparing traces of Frobenius,
that for primes $p$
\[a_p(E_0) \equiv a_p(E) \pmod{l} \quad \mathrm{if\ } p \nmid N\]
\[a_p(E_0) \equiv a_p(E) (1+p) \equiv \pm (1+p) \pmod{l} \quad \mathrm{if\ } p | N \mathrm{\ and\ } p\nmid N_0=2^5 \cdot 3.\]
Note that we explicitly know $E_0$ up to isogeny and quadratic
twist. So if $E_0'$ is any elliptic curve over $\Q$ with conductor
$96$, e.g. given by $Y^2=X^3\pm X^2-2X$, then $a_p(E_0)^2 =
a_p(E_0')^2$ for all primes $p$. Therefore, for every prime $p$,
we can effectively compute the uniquely determined value
$a_p(E_0)^2$. We summarize some of the Diophantine information
obtained so far.
\begin{lemma}\label{lem Modular approach}
Suppose that $(x,y,z)$ is a solution to (\ref{eqn 2_2l_3}) for
$n=l$, let $u,v$ be as in Lemma \ref{lem descent}, let the
elliptic curve $E$ be given by (\ref{eqn Frey Curve}), let $E_0$
be any elliptic curve over $\Q$ with conductor $96$ and let $p>3$
be prime. If $\rho_l^E$ is irreducible and $a_p(E_0)^2 \not \equiv
(p+1)^2 \pmod{l}$, then $p\nmid v(3u^2-v^2)$ and $a_p(E_0)^2
\equiv a_p(E)^2 \pmod{l}$.
\end{lemma}

We see that in order to obtain a contradiction and conclude that
(\ref{eqn 2_2l_3}) has no solutions for $n=l$, it suffices (still
assuming the irreducibility of $\rho_l^E$) to find a prime $p>3$
such that $a_p(E_0)^2 \not\equiv (p+1)^2 \pmod{l}$ and $a_p(E_0)^2
\not\equiv a_p(E)^2 \pmod{l}$. A priori possible values of
$a_p(E)$ can be obtained by plugging all values of $u,v \pmod{p}$
with $p\nmid v(3u^2-v^2)$ into the equation for $E$. However, this
will never lead to a contradiction, since plugging the values
$u,v$ with $(|u|,|v|)=(2,3)$ into the equation for $E$, will give
us elliptic curves with conductor $96$. Now by reducing (\ref{eqn
descent 2}) modulo $p$, one obtains extra information on $u,v
\pmod{p}$ whenever one knows that the nonzero $l$-th powers modulo
$p$ are strictly contained in $\F_p^*$. This happens exactly when
$p \equiv 1 \pmod{l}$, say $p=kl+1$, in which case the nonzero
$l$-th powers modulo $p$ are given by
\[\mu_k(\F_p):=\set{\zeta \in \F_p}{\zeta^k=1}.\]
Heuristically speaking, the bigger $k$ is, the more possibilities
for $u,v \pmod{p}$ we expect, hence the more a priori
possibilities for $a_p(E)$ we expect, hence the less likely it is
to conclude that $a_p(E)^2 \not\equiv a_p(E_0)^2 \pmod{l}$. From a
computational point of view, it is desirable to only consider the
possibilities for $u/v \pmod{p}$, this indeed suffices, since $u/v
\pmod{p}$ determines the reduction of $E$ modulo $p$, denoted
$\tilde{E}(\F_p)$, up to quadratic twist and hence determines
$a_p(E)^2$ uniquely.

Continuing our more formal discussion, let indeed $k \in \N$ be
such that $p:=kl+1$ is prime (so $(p+1)^2 \equiv 4 \pmod{l}$) and
suppose that $a_p(E_0)^2 \not\equiv 4 \pmod{l}$. Then in
particular $p\nmid v(3u^2-v^2)=(rs)^l$. Define $U:=u/(3v)$. From
$3v=r^l$ and $3u^2-v^2=3s^l$, we get $U^2=1/(27)+(s/r^2)^l$. Since
$p \nmid rs$, we obtain for the reduction of $U$ modulo $p$,
denoted $\overline{U}$, that
\begin{equation}\label{eqn Skp}
\overline{U} \in S_{k,p}:=\set{\alpha \in
\F_p}{\alpha^2-\frac{1}{27}\in \mu_k(\F_p)}.
\end{equation}
For $\alpha \in S_{k,p}$ consider the elliptic curve over $\F_p$
given by
\begin{equation}\label{eqn Ealpha}
E_{\alpha}: Y^2=X^3+2\alpha X^2+\frac{1}{27}X.
\end{equation}
Note that $E_{\overline{U}}$ is a quadratic twist of
$\tilde{E}(\F_p)$, so that $a_p(E_{\overline{U}})^2 = a_p(E)^2$.
If now for all $\alpha \in S_{k,p}$ we have $a_p(E_{\alpha})^2
\not\equiv a_p(E_0)^2 \pmod{l}$, then $a_p(E)^2  \equiv
a_p(E_{\overline{U}})^2 \not\equiv a_p(E_0)^2 \pmod{l}$, a
contradiction which implies that (\ref{eqn 2_2l_3}) has no
solutions for $n=l$. Note that we have assumed $\rho_l^E$ to be
irreducible; one readily obtains (see Section \ref{sec
Irreducibility}) from \cite{Mazur 78} and \cite{Mazur Velu 72}
that this is actually the case if $l>7$. This proves the main
theorem of \cite{Chen 08}.

\begin{theorem}[{\cite[Theorem 1]{Chen 08}}]\label{thm Chen}
Let $l>7$ be prime and let $E_0$ be any elliptic curve over $\Q$
with conductor $96$. If there exists a $k \in \Z_{>0}$ such that
the following three conditions hold
\begin{enumerate}
\item $p:=kl+1$ is prime \item\label{item 2} $a_p(E_0)^2
\not\equiv 4 \pmod{l}$ \item\label{item 3} $a_p(E_{\alpha})^2
\not\equiv a_p(E_0)^2 \pmod{l}$ for all $\alpha \in S_{k,p}$,\\
where $S_{k,p}$ and $E_{\alpha}$ are given by (\ref{eqn Skp}) and
(\ref{eqn Ealpha}) respectively,
\end{enumerate}
then (\ref{eqn 2_2l_3}) has no solutions for $n=l$.
\end{theorem}

According to \cite{Chen 08}, it has been computationally verified
that for every prime $l$ with $7<l<10^7$ and $l\not=31$ there
exists a $k\in \Z_{>0}$ satisfying the three conditions of Theorem
\ref{thm Chen}, hence (\ref{eqn 2_2l_3}) has no solutions for
$n=l$ if $l$ equals one of these values.

We like to take this opportunity to point out a small omission in
the proof of \cite[Corollary 3]{Chen 08}. This corollary to
Theorem \ref{thm Chen} states that if $l>7$ is a Sophie Germain
prime, i.e. $p:=2l+1$ is prime, $\left(\frac{p}{7}\right)=1$ and
$\left(\frac{p}{13}\right)=(-1)^{(l+1)/2}$, then (\ref{eqn
2_2l_3}) has no solutions for $n=l$. As pointed out in \cite{Chen
08}, the conditions for the Legendre symbols are equivalent (by
quadratic reciprocity) to $S_{2,p}$ being empty, in which case
condition \ref{item 3} of Theorem \ref{thm Chen} trivially holds.
However, in order to use Theorem \ref{thm Chen} to deduce that
(\ref{eqn 2_2l_3}) has no solutions, we need of course prove that
condition \ref{item 2} also holds. This can be done as follows.
Note that both isogeny classes of elliptic curves over $\Q$ with
conductor $96$ contain an elliptic curve with rational torsion
group of order $4$. So $4|p+1-a_p(E_0)$, and hence $4|a_p(E_0)$.
This implies that if $a_p(E_0) \equiv \pm 2 \pmod{l}$, then
$|a_p(E_0) \mp 2| \geq 2l$. That this last inequality cannot hold,
follows directly from the Hasse bound $|a_p(E_0)| \leq 2 \sqrt{p}$
together with $l>3$, which completes the proof of the corollary.

Note that the corollary and its proof are analogously to work in
\cite{Kraus 98}, namely Corollaire 3.2 and its proof.

\section{Refinements}\label{sec Refinements}

In this section we shall prove our main result, Theorem \ref{thm
main}.

\subsection{Irreducibility.}\label{sec Irreducibility}

As pointed out in \cite{Chen 08}, since $E$ has at least one odd
prime of multiplicative reduction and a rational point of order
$2$, the irreducibility of $\rho_l^E$ for primes $l>7$ follows
from \cite[Corollary 4.4]{Mazur 78} and \cite{Mazur Velu 72}. In
fact, one does not need the first mentioned property of $E$, but
only that $E$ has a rational point of order $2$ to arrive at the
desired conclusion, since it is well-known that the modular curves
$X_0(2l)$ for primes $l>7$ have no noncuspidal rational points.
Irreducibility of $\rho_l^E$ for the remaining primes $l>3$ can
also be obtained.

\begin{theorem}\label{thm Irreducibility}
Let $l>3$ be prime and consider the elliptic curve $E$ given by
(\ref{eqn Frey Curve}), where $u,v \in \Z$ with
$v(3u^2-v^2)\not=0$. Then $\rho_l^E$ is irreducible.
\end{theorem}

\begin{proof}
The only cases left to deal with are $l=7$ and $l=5$.

The modular curve $X_0(14)$ is of genus one, it has $2$
noncuspidal rational points, which correspond to elliptic curves
with $j$-invariant $j_{14}$ equal to $-3^3 \cdot 5^3$ or $3^3
\cdot 5^3\cdot 17^3$; this readily follows from \cite[Chapter
5]{Ligozat}. Denote the $j$-invariant of $E$ by $j(u,v)$. For both
values of $j_{14}$ it is completely straightforward to check that
the equation $j(u,v)=j_{14}$ has no solutions with $u,v \in \Z$
and $(3u^2-v^2)v\not=0$. This proves that $\rho_l^E$ is
irreducible for $l=7$.

Since the modular curve $X_0(10)$ is of genus zero and has a
rational cusp (which is nonsingular of course), it has infinitely
many noncuspidal rational points. So we have to work a little
harder to obtain the irreducibility of $\rho_l^E$ for $l=5$. Up to
quadratic twist, the Frey curve $E$ is $2$-isogenous to
\[E': Y^2=X^3-6uX^2+3(3u^2-v^2)X\]
(replacing $X$ by $X+2u$, gives us back the model for $E'$ given
in Section \ref{sec A modular approach}). So $\rho_5^E$ is
irreducible if and only if $\rho_5^{E'}$ is irreducible. The
$j$-invariant $j'(u,v)$ of $E'$ (which is a twist of the $j$-map
from $X(2)$ to $X(1)$) is given by
\begin{eqnarray*}
j'(u,v) & = & \frac{1728(u^2 + v^2)^3}{v^2(3u^2 - v^2)^2}\\
 & = & \frac{1728 u^2 (u^2 - 3v^2)^2}{v^2(3 u^2 - v^2)^2}+1728.
 \end{eqnarray*}
The $j$-map from the modular curve $X_0(5)$ to $X(1)$, denoted
$j_5$, is given by
\begin{eqnarray*}
j_5(s,t) & = & \frac{(t^2+10st+5s^2)^3}{s^5 t} \\
 & = & \frac{(t^2+4st-s^2)^2(t^2+22st+125s^2)}{s^5t}+1728.
\end{eqnarray*}
By comparing $j-1728$, we see that every $[u:v] \in \PP^1(\Q)$
such that $\rho_5^{E'}$ is irreducible gives rise to a rational
point on the curve $C$ in $\PP^1 \times \PP^1$ give by
\[
C: \frac{1728 u^2 (u^2 - 3v^2)^2}{v^2(3 u^2 -
v^2)^2}=\frac{(t^2+4st-s^2)^2(t^2+22st+125s^2)}{s^5t}.
\]
Letting
\[
X:=\frac{t}{s},\quad Y:=\frac{u(u^2-3v^2)}{v(3u^2-v^2)} \cdot
\frac{st}{t^2+4st-s^2}
\]
defines a covering by $C$ of the elliptic curve over $\Q$ given by
\[1728Y^2=(X^2+22X+125)X.\]
By checking that this elliptic curve has rank $0$ and torsion
subgroup of order $2$, we get that the only rational points on $C$
are those with $[s:t]=[1:0]$ or $[s:t]=[0:1]$. For these values of
$[s:t]$ we have $j_5(s,t)=\infty$, so they correspond to cusps on
$X_0(5)$. We conclude that $\rho_l^{E'}$, and hence $\rho_l^E$, is
irreducible for $l=5$. This completes the proof.
\end{proof}

This irreducibility result obviously leads to the following
strengthening of Theorem \ref{thm Chen}.

\begin{theorem}\label{thm Chen refined l>3}
Theorem \ref{thm Chen} holds true with the condition $l>7$
replaced by $l>3$.
\end{theorem}

\subsubsection{$l=5$.}

In order to show that (\ref{eqn 2_2l_3}) has no solutions for
$n=l:=5$, it suffices by Theorem \ref{thm Chen refined l>3} to
show that there exists a $k \in \N$ satisfying the three
conditions of Theorem \ref{thm Chen}. Let $k:=2$, then
$p:=kl+1=11$ is prime, $a_p(E_0)^2 \equiv 1 \not\equiv 4 \pmod{l}$
and $S_{2,p}$ is empty because $1/27\pm 1$ are not squares in
$\F_p$. So the conditions are satisfied for $k=2$ and we conclude
that this proves Theorem \ref{thm main} in case $n=l=5$.

\begin{remark}
By our irreducibility result, Theorem \ref{thm Irreducibility},
the corollary mentioned at the end of Section \ref{sec A modular
approach} also holds with the condition $l>7$ replaced by $l>3$
(the case $l=7$ is in fact trivial, because $7$ is not a Sophie
Germain prime). So for obtaining Theorem \ref{thm main} in case
$n=l:=5$ it sufficed to check that $\left(\frac{p}{7}\right)=1$
and $\left(\frac{p}{13}\right)=(-1)^{(l+1)/2}$ for $p=11$.
\end{remark}

\subsection{Using more local information.}

For $l:=31$ there is no $k \in \Z_{>0}$ known which satisfies the
three conditions of Theorem \ref{thm Chen} (it seems in fact very
unlikely that such a $k$ exists), so that we cannot use this
theorem to prove the nonexistence of solutions of (\ref{eqn
2_2l_3}) for $n=l=31$. This is not too surprising, since, loosely
speaking, $l=31$ is very far from being a Sophie Germain prime in
the sense that the smallest $k \in \Z_{>0}$ such that $kl+1$ is
prime, which is $k=10$, is not so small compared to $l=31$.

Now let $l>3$ be prime, suppose that $(x,y,z)$ is a solution to
(\ref{eqn 2_2l_3}) for $n=l$ and let $u,v,r,s$ be as in Lemma
\ref{lem descent}. The idea is to factor the left-hand side of
$3u^2-v^2=3s^l$ over the ring of integers $R:=\Z[\sqrt{3}]$ in
order to obtain more local information on $u,v$, which, together
with Lemma \ref{lem Modular approach} should lead to a
contradiction. We claim that
\begin{equation}\label{eqn descent 3}
3v=r^l, \quad \sqrt{3}u-v=\sqrt{3}x_1^l \epsilon \quad
\mathrm{and} \quad \sqrt{3}u+v=\sqrt{3}x_2^l \epsilon^{-1}
\end{equation}
for some nonzero $x_1,x_2 \in R$ and unit $\epsilon \in R^*$. Note
that $R$ has class number one. Using $\gcd(u,v)=1$, $2|uv$ and
$3|v$ we readily get $(\sqrt{3}u-v,\sqrt{3}u+v)=(\sqrt{3})$ and
$\sqrt{3}||\sqrt{3}u \pm v$. From
$(\sqrt{3}u-v)(\sqrt{3}u+v)=3s^l$ the claim now follows.

As before, let $E$ be given by (\ref{eqn Frey Curve}), let $E_0$
be any elliptic curves over $\Q$ with conductor $96$, let $k \in
\N$ be such that $p:=kl+1$ is prime and suppose that $a_p(E_0)^2
\not\equiv 4 \pmod{l}$. Note that by Theorem \ref{thm
Irreducibility} we get that $\rho_l^E$ is irreducible. Again, by
Lemma \ref{lem Modular approach} we get $p\nmid
v(3u^2-v^2)=(rs)^l$ and we see that if we can show that
$a_p(E_0)^2 \not\equiv a_p(E)^2 \pmod{l}$, then we reach a
contradiction which shows that (\ref{eqn 2_2l_3}) has no solutions
for $n=l$. Now suppose furthermore that $p$ splits in $R$. Denote
by $\Ideal{p}$ any of the 2 primes of $R$ lying above $p$, for any
$x\in R$ denote by $\overline{x}$ the reduction of $x$ modulo
$\Ideal{p}$ in $R/\Ideal{p} \simeq \F_p$ and set
$r_3:=\overline{\sqrt{3}}$. By reducing (\ref{eqn descent 3})
modulo $\Ideal{p}$ we get
\[
3\overline{v}=\zeta_0, \quad r_3 \overline{u}-\overline{v}=r_3
\zeta_1 \overline{\epsilon}, \quad \mathrm{and} \quad r_3
\overline{u}+\overline{v}=r_3 \zeta_2 \overline{\epsilon}^{-1}
\]
for some $\zeta_0,\zeta_1,\zeta_2 \in \mu_k(\F_p)$. Let
$U:=u/(3v)$ as before, set $\zeta_1':=\zeta_1/\zeta_0,
\zeta_2':=\zeta_2/\zeta_0$ and divide by $3 r_3 \overline{v}=r_3
\zeta_0$ to obtain
\begin{equation}\label{eqn U bar}
\overline{U}-\frac{1}{3r_3}=\zeta_1'\overline{\epsilon},\quad
\overline{U}+\frac{1}{3r_3}=\zeta_2'\overline{\epsilon}^{-1}.
\end{equation}
By Dirichlet's unit theorem $R^*= \langle -1,\epsilon_f \rangle$
for some fundamental unit $\epsilon_f \in R$ (we can take for
example $\epsilon_f=2+\sqrt{3}$). So $\zeta_1'\overline{\epsilon},
\zeta_2'\overline{\epsilon}^{-1}$ belong to the subgroup of
$\F_p^*$ generated by $\mu_k(\F_p)$ and $\overline{\epsilon_f}$,
which we denote by $G_{k,p}$. If $\overline{\epsilon_f} \not \in
\mu_k(\F_p)$ one easily obtains that $G_{k,p}=\F_p^*$. Together
with (\ref{eqn U bar}) this only gives us back the original
information $\overline{U} \in S_{k,p}$. If however
$\overline{\epsilon_f} \in \mu_k(\F_p)$, then of course
$G_{k,p}=\mu_k(\F_p)$ and we get
\begin{equation}\label{eqn Skp bis}
\overline{U} \in S_{k,p}':=(\mu_k(\F_p) +\frac{1}{3r_3}) \cap
(\mu_k(\F_p) -\frac{1}{3r_3}).
\end{equation}
This might be much stronger information, because possibly (and
heuristically speaking, for large $k$ even very likely) $S_{k,p}'$
is strictly contained in $S_{k,p}$. So suppose
$\overline{\epsilon_f}^k=1$. Since
$a_p(E)^2=a_p(E_{\overline{U}})^2$, we reach a contradiction if
for all $\alpha \in S_{k,p}'$ we have $a_p(E_{\alpha})^2
\not\equiv a_p(E_0)^2 \pmod{l}$. We arrive at the following.

\begin{theorem}\label{thm sqrt 3 Kraus}
Let $l>3$ be prime and let $E_0$ be any elliptic curve over $\Q$
with conductor $96$. If there exists a $k \in \Z_{>0}$ such that
the following five conditions hold
\begin{enumerate}
\item $p:=kl+1$ is prime \item $p$ splits in $\Z[\sqrt{3}]$
\item\label{item fundamental unit}
$p|\Norm_{\Q(\sqrt{3})/\Q}\left((2+\sqrt{3})^k-1\right)$ \item
$a_p(E_0)^2 \not\equiv 4 \pmod{l}$ \item $a_p(E_{\alpha})^2
\not\equiv a_p(E_0)^2 \pmod{l}$ for all $\alpha \in S_{k,p}'$,\\
where $S_{k,p}'$ and $E_{\alpha}$ are given by (\ref{eqn Skp bis})
and (\ref{eqn Ealpha}) respectively,
\end{enumerate}
then (\ref{eqn 2_2l_3}) has no solutions for $n=l$.
\end{theorem}

\begin{remark}
Suppose that $k \in \N$ satisfies the first two conditions of
Theorem \ref{thm Chen}. Heuristically speaking, by considering the
expected size of $S_{k,p}$, it seems highly unlikely that if $k
\gtrsim l$, the last condition is satisfied. Now condition
\ref{item fundamental unit} of Theorem \ref{thm sqrt 3 Kraus} is
very restrictive. But if $k \in \N$ does satisfy the first four
conditions of Theorem \ref{thm sqrt 3 Kraus}, then the expected
size of $S'_{k,p}$ is much smaller then that of $S_{k,p}$ and it
only becomes highly unlikely that the last condition is satisfied
when $k \gtrsim l^2$. This is of course all very rough, but the
main idea about the  benefits of Theorem \ref{thm sqrt 3 Kraus}
are hopefully clear.
\end{remark}

\subsubsection{$l=7$.}

Let $l:=7$. There are no problems with irreducibility. However, a
large computer search did not reveal a $k \in \N$ satisfying the
three conditions of Theorem \ref{thm Chen} or the five conditions
of Theorem \ref{thm sqrt 3 Kraus} (and we believe that it is very
unlikely that such a $k \in \N$ exists).

Although the modular methods can give us many congruence relations
for a hypothetical solution $(x,y,z)$ of (\ref{eqn 2_2l_3}) with
$n=l=7$, such as $2|x$ and $29|y$, we are not able to show that
(\ref{eqn 2_2l_3}) has no solutions for $n=l=7$ along these lines.
Anyway, in \cite{PSS 07} all finitely many nonzero coprime $a,b,c
\in \Z$ satisfying $a^2+b^3=c^7$ are determined and one readily
checks that for none of the solutions $-c$ is a square. So we
conclude that (\ref{eqn 2_2l_3}) has no solutions for $n=l=7$, as
mentioned before.

\subsubsection{$l=31$.}

Now let $l:=31$, $k:=718$ and $p:=kl+1=22259$. Then one quickly
verifies that $p$ is a prime that splits in $\Z[\sqrt{3}]$ and
that $\overline{\epsilon_f}^k=1$. We calculate $a_p(E_0)=\pm 140$,
so $a_p(E_0)^2 \equiv 8 \not\equiv 4 \pmod{l}$. Finally, the set
$S_{k,p}'$ is easily determined explicitly. All the elements
$\alpha \in S_{k,p}'$ are given in the first column of Table
\ref{Table U ap} together with the corresponding values of
$a_p(E_{\alpha})^2 \pmod{l}$.

\begin{table}[ht]
\begin{center}
\begin{tabular}{c|c}
$\pm \alpha$ & $a_p(E_{\alpha})^2 \pmod{l}$\\
\hline
127 & 28\\
1852 & 19\\
2818 & 1\\
3146 & 10\\
3615 & 9\\
3764 & 16\\
4419 & 18\\
5889 & 25\\
7994 & 20\\
8058 & 0\\
8330 & 19\\
10171 & 18\\
10561 & 5\\
\end{tabular}
\caption{elements of $S_{k,p}'$ with corresponding values of
$a_p(E_{\alpha})^2 \pmod{l}$}\label{Table U ap}
\end{center}
\end{table}

Recall that $a_p(E_0)^2 \equiv 8 \pmod{l}$, so from the second
column of Table \ref{Table U ap} we see that $a_p(E_{\alpha})^2
\not\equiv a_p(E_0)^2 \pmod{l}$ for all $\alpha \in S_{k,p}'$. By
Theorem \ref{thm sqrt 3 Kraus} we can now conclude that (\ref{eqn
2_2l_3}) has no solutions for $n=l=31$. This concludes the proof
of Theorem \ref{thm main}.

\begin{remark}
For $l=31$ the smallest $k \in \N$ satisfying the five conditions
of Theorem \ref{thm sqrt 3 Kraus} is $k=718$. Another $k \in \N$
satisfying these conditions is $k=2542$, and it seems very likely
that there are no other such $k$.
\end{remark}

\section{Extra information from quadratic reciprocity}

In this section we shall prove Theorem \ref{thm n congruence}.

In \cite[Section 4]{Chen Siksek 09}, quadratic reciprocity (over
$\Q$) is used to obtain that there are no nonzero coprime integers
$a,b,c$ satisfying $a^3+b^3=c^n$ for $n \in \N$ with $n\equiv 51,
103 \mathrm{\ or\ } 105 \pmod{106}$. The same method can be
applied to (\ref{eqn 2_2l_3}). The key information obtained is the
following.

\begin{lemma}\label{lem quadratic reciprocity}
Let $l>3$ be prime, suppose that $(x,y,z)$ is a solution to
(\ref{eqn 2_2l_3}) for $n=l$ and let $r,s$ be as in Lemma \ref{lem
descent}. Then $s-r^2$ is a square modulo $7$.
\end{lemma}

\begin{proof}
Let $u,v$ be as in Lemma \ref{lem descent}. Using (\ref{eqn
descent 2}) we get
\begin{equation}\label{eqn quadratic}
s-r^2|3(s^l-r^{2l})=3u^2-28v^2.
\end{equation}
Suppose that $q$ is an odd prime that divides $s-r^2$. Then
$(3u)^2 \equiv 3 \cdot 7 (2v)^2 \pmod{q}.$ From $3v=r^l$ and
$\gcd(r,s)=1$ we obtain $q\nmid 2v$, so $3 \cdot 7$ is a square
modulo $q$, i.e. $\left(\frac{3\cdot7}{q}\right) \in \{0,1\}$. By
quadratic reciprocity we obtain $\left(\frac{q}{3\cdot7}\right)
\in \{0,1\}$. We have $\left(\frac{-1}{3\cdot7}\right)=1$. However
$\left(\frac{2}{3\cdot7}\right)=-1$, but we claim that
$\ord_2(s-r^2)$ is even, so that we get
$\left(\frac{s-r^2}{3\cdot7}\right) \in \{0,1\}$. We claim
furthermore that $\left(\frac{s-r^2}{3}\right)=1$, from which it
now follows that $s-r^2$ is a square modulo $7$. It remain to
prove our two claims. By Theorem \ref{thm Irreducibility} and the
discussion in Section \ref{sec A modular approach}, we know that
$2|u$, so $2\nmid v$. Suppose first that $2||u$, then
$3(u/2)^2-7v^2 \equiv 4 \pmod{8}$, hence
$\ord_2(s^l-r^{2l})=\ord_2(3u^2-28v^2)=4$. Next suppose that
$4|u$, then we immediately get
$\ord_2(s^l-r^{2l})=\ord(3u^2-28v^2)=2$. So in any case,
$\ord_2(s^l-r^{2l})$ is even. From the fact that $r,s$ are both
odd and by counting terms in $(s^l-r^{2l})/(s-r^2)$, we get that
the latter quantity is odd. We conclude that
$\ord_2(s-r^2)=\ord_2(s^l-r^{2l})$ is even, which proves our first
claim. For our second claim, note that $3|r, 3\nmid s$ and $l$ is
odd, so it remains to prove that $s^l$ is a square modulo $3$.
From $3|v$ we get $3|v^2/3$, so from (\ref{eqn descent 2}) we get
$s^l=u^2-v^2/3\equiv u^2 \pmod{3}$. This proves our second claim,
thereby finishing the proof of the lemma.
\end{proof}

From Lemma \ref{lem Modular approach} we obtain some information
on $(s/r^2)^l$ modulo $7$ in a straightforward way.

\begin{lemma}\label{lem straightforward}
Let $l>3$ be prime, suppose that $(x,y,z)$ is a solution to
(\ref{eqn 2_2l_3}) for $n=l$ and let $r,s$ be as in Lemma \ref{lem
descent}. Then $7 \nmid r$ and $(s/r^2)^l \equiv 2 \pmod{7}$
\end{lemma}

\begin{proof}
We shall use Lemma \ref{lem Modular approach} for $p=7$, so let
$u,v,E,E_0$ be as in that lemma. We calculate $a_7(E_0)=\pm 4$, so
$a_7(E_0) \not\equiv \pm 8 \pmod{l}$. Together with the
irreducibility of $\rho_l^E$ form Theorem \ref{thm
Irreducibility}, we get that $7 \nmid v(3u^2-v^2)=(rs)^l$ and
$a_7(E_0)^2 \equiv a_7(E)^2 \pmod{l}$. Let $U:=u/(3v)$ as before
and $\overline{U}$ the reduction of $U$ modulo $7$. Then
$a_7(E_0)^2 \equiv a_7(E_{\overline{U}})^2 \pmod{l}$, where
$E_{\alpha}$ denotes the elliptic curve over $\F_7$ given by
(\ref{eqn Ealpha}) for $\alpha \in \F_7$. We compute that for
$\alpha \not= \pm 1$ we have $a_7(E_0)^2-a_7(E_{\alpha})^2 \in
\{12,16\}$ (for $\alpha=\pm 1$ we have in fact
$a_7(E_0)^2=a_7(E_{\alpha})^2$). We see that if $\alpha \not= \pm
1$, then $a_7(E_0)^2 \not\equiv a_7(E_{\alpha})^2 \pmod{l}$.
Together with $a_7(E_0)^2 \equiv a_7(E_{\overline{U}})^2 \pmod{l}$
we obtain that $\overline{U}=\pm 1$. From (\ref{eqn descent 2}) we
get that
\[\left(\frac{s}{r^2}\right)^l=\frac{3u^2-v^2}{27v^2}=U^2-\frac{1}{27}.\]
Reducing this modulo $7$ and using $\overline{U}=\pm 1$, we arrive
at $(s/r^2)^l \equiv 2 \pmod{7}$.
\end{proof}

Combining the two lemmas above readily leads to the following.

\begin{theorem}\label{thm congruence l}
Let $l$ be prime with $l \equiv -1 \pmod{6}$. Then (\ref{eqn
2_2l_3}) has no solutions for $n=l$.
\end{theorem}

\begin{proof}
Suppose that $(x,y,z)$ is a solution to (\ref{eqn 2_2l_3}) for
$n=l$ and let $r,s$ be as in Lemma \ref{lem descent}. From Lemma
\ref{lem straightforward} and $l\equiv -1 \pmod{6}$, we get
$(s/r^2)^l \equiv (s/r^2)^{-1} \equiv 2 \pmod{7}$. This gives
$s/r^2-1 \equiv 2^{-1}-1 \equiv 3 \pmod{7}$. From Lemma \ref{lem
quadratic reciprocity} we get that $s/r^2-1$ is a square modulo
$7$. However, $3$ is not a square modulo $7$, a contradiction
which proves the theorem.
\end{proof}

Theorem \ref{thm n congruence} now follows by noting that if $n
\equiv -1 \pmod{6}$, then $n$ is divisible by a prime $l \equiv -1
\pmod{6}$.

\begin{remark}
If $l \equiv 1 \pmod{6}$, then Lemma \ref{lem straightforward}
leads to $s/r^2-1 \equiv 1 \pmod{7}$, and since $1$ is a square in
$\F_7$ we do not get a contradiction together with Lemma \ref{lem
quadratic reciprocity}. Furthermore, a similar relation as
(\ref{eqn quadratic}) is given by
$s+r^2|3(s^l+r^{2l})=3u^2+26v^2$, however, similar arguments as
above do not work in this case also.

Finally, in \cite{Chen Siksek 09} not only quadratic reciprocity
over $\Q$ is applied to obtain results about the generalized
Fermat equation $a^3+b^3=c^n$, but also quadratic reciprocity over
(other) number fields. Applying similar methods in our situation
is not so straightforward. This is because $a^3+b^3$ splits into
linear factors in an imaginary quadratic number field, whose ring
of integers has finite unit group, whereas $v(3u^2-v^2)$ splits
into linear factors in a real quadratic number field, whose ring
of integers has infinite unit group. Note that we already dealt
with this infinite unit group in one particular situation, namely
when we solved (\ref{eqn 2_2l_3}) for $n=31$. Trying to approach
(\ref{eqn 2_2l_3}) using quadratic reciprocity over number fields
might be interesting future research.
\end{remark}

\section*{Acknowledgements}
This article contains material from Chapter 3 of the author's
Ph.D. thesis \cite{Dahmen 08}. The author would like to thank
Frits Beukers for many useful discussions related to this
material.

\bibliographystyle{elsarticle-num}

\end{document}